\documentclass{article}

\usepackage{amsthm,amssymb,amsmath}
\usepackage{a4wide}
\usepackage{cite}

\newtheorem{theorem}{Theorem}
\newtheorem{corollary}[theorem]{Corollary}
\newtheorem{proposition}[theorem]{Proposition}
\newtheorem{definition}[theorem]{Definition}
\newtheorem{example}[theorem]{Example}
\newtheorem{remark}[theorem]{Remark}

\begin{document}

\title{A Conformable Fractional Calculus on Arbitrary Time Scales\thanks{This 
is a preprint of a paper whose final and definite form is in \emph{Journal 
of King Saud University (Science)}, ISSN 1018-3647. Paper submitted 14/April/2015; 
revised 12/May/2015; accepted for publication 12/May/2015.}}

\author{Nadia Benkhettou$^1$\\ {\small \texttt{benkhettou$_{-}$na@yahoo.fr}}
\and Salima Hassani$^1$\\ {\small \texttt{salima$_{-}$hassani@yahoo.fr}}
\and Delfim F. M. Torres$^2$\thanks{Corresponding author.
Tel: +351 234370668; Fax: +351 234370066;  Email: delfim@ua.pt}\\ {\small \texttt{delfim@ua.pt}}}

\date{$^1$Laboratoire de Bio-Math\'{e}matiques, Universit\'{e} de Sidi Bel-Abb\`{e}s,\\
B.P. 89, 22000 Sidi Bel-Abb\`{e}s, Algerie\\[0.3cm]
$^2$\text{Center for Research and Development in Mathematics and Applications (CIDMA)},\\
Department of Mathematics, University of Aveiro, 3810-193 Aveiro, Portugal}

\maketitle


\begin{abstract}
A conformable time-scale fractional calculus of order $\alpha \in ]0,1]$
is introduced. The basic tools for fractional differentiation and fractional
integration are then developed. The Hilger time-scale calculus is obtained
as a particular case, by choosing $\alpha = 1$.

\smallskip


\smallskip

\end{abstract}


\section{Introduction}

Fractional calculus is nowadays one of the most intensively
developing areas of mathematical analysis \cite{MyID:310,MR2736622,Tarasov},
including several definitions of fractional operators like Riemann--Liouville,
Caputo, and Gr\"{u}nwald--Letnikov.
Operators for fractional differentiation and integration
have been used in various fields, such as signal processing, hydraulics of dams,
temperature field problem in oil strata, diffusion problems,
and waves in liquids and gases \cite{BBT,Boya-Scher,Sch-Wy}.
Here we introduce the notion of conformable fractional
derivative on a time scale $\mathbb{T}$. The notion of
conformable fractional derivative in $\mathbb{T}=[0, \infty)$
is a recent one: it was introduced in \cite{KhHYS},
then developed in \cite{Abdeljawad}, and is currently under
intensive investigations \cite{MR3326681}.
In all these works, however, only the case
$\mathbb{T}=[0, \infty)$ is treated, providing a natural extension
of the usual derivative. In contrast, here we introduce
the conformable natural extension of the time-scale derivative.
A time scale $\mathbb{T}$ is an arbitrary nonempty closed subset of $\mathbb{R}$.
It serves as a model of time. The calculus on time scales was
initiated by Aulbach and Hilger in 1988 \cite{MR1062633}, in order
to unify and generalize continuous and discrete analysis \cite{H2,H1}.
It has a tremendous potential for applications and has
recently received much attention \cite{ABRP}. The reader interested
on the subject of time scales is referred  to the books \cite{BP,BP1}.

The paper is organized as follows. In Section~\ref{sub:sec:FD},
the conformable fractional derivative for functions defined
on arbitrary time scales is introduced, and the respective conformable
fractional differential calculus developed. Then, in Section~\ref{sub:sec:FI},
we introduce the notion of conformable fractional integral on time scales
(the $\alpha$-fractional integral) and investigate some of its basic properties.
We end with Section~\ref{sec:conc} of conclusion.


\section{Conformable Fractional Differentiation}
\label{sub:sec:FD}

Let $\mathbb{T}$ be a time scale, $t\in \mathbb{T}$,
and $\delta >0$. We define the $\delta$-neighborhood of $t$
as $\mathcal{V}_t :=\left] t-\delta ,t+\delta\right[ \cap \mathbb{T}$.
We begin by introducing a new notion:
the conformable fractional derivative of order $\alpha \in ]0,1]$
for functions defined on arbitrary time scales.

\begin{definition}
\label{def:fd:ts}
Let $f:\mathbb{T}\rightarrow \mathbb{R}$, $t\in \mathbb{T}^{\kappa }$,
and $\alpha \in ]0,1]$. For $t>0$, we define
$T_\alpha(f)(t)$ to be the number (provided it exists) with the property
that, given any $\epsilon >0$, there is a $\delta$-neighborhood
$\mathcal{V}_t\subset \mathbb{T}$ of $t$, $\delta > 0$, such that
$\left \vert \left[ f(\sigma (t))-f(s)\right]t^{1-\alpha}
-T_\alpha(f)(t)\left[ \sigma(t)-s\right]\right \vert
\leq \epsilon \left \vert \sigma
(t)-s\right \vert$
for all $s\in \mathcal{V}_t$. We call
$T_\alpha(f)(t)$ the conformable fractional derivative of $f$
of order $\alpha $ at $t$, and we define the conformable
fractional derivative at 0 as
$T_\alpha(f)(0)=\displaystyle\lim_{t\rightarrow 0^+} T_\alpha(f)(t)$.
\end{definition}

\begin{remark}
If $\alpha = 1$, then we obtain from Definition~\ref{def:fd:ts}
the delta derivative of time scales.
The conformable fractional derivative of order zero
is defined by the identity operator: $T_0(f) := f$.
\end{remark}

\begin{remark}
Along the work, we also use the notation
$\left(f(t)\right)^{(\alpha)} = T_\alpha(f)(t)$.
\end{remark}

The next theorem provides some useful relationships
concerning the conformable fractional derivative on time scales
introduced in Definition~\ref{def:fd:ts}.

\begin{theorem}
\label{T1}
Let $\alpha \in ]0,1]$ and $\mathbb{T}$ be a time scale.
Assume $f:\mathbb{T}\rightarrow \mathbb{R}$ and let
$t\in \mathbb{T}^{\kappa }$. The following properties hold.
(i) If $f$ is conformal fractional differentiable of order
$\alpha$ at $t>0$, then $f$ is continuous at $t$.
(ii) If $f$ is continuous at $t$ and $t$ is right-scattered, then $f$
is conformable fractional differentiable of order $\alpha$ at $t$ with
\begin{equation}
\label{eq:delf:dr}
T_\alpha(f)(t)=\frac{f(\sigma (t))-f(t)}{\mu (t)}t^{1-\alpha }.
\end{equation}
(iii) If $t$ is right-dense, then $f$ is conformable fractional
differentiable of order $\alpha$ at $t$ if, and only if, the limit
$\lim_{s\rightarrow t}\frac{f(t)-f(s)}{(t-s)}t^{1-\alpha}$
exists as a finite number. In this case,
\begin{equation}
\label{eq:rhsof}
T_{\alpha}(f)(t)=\lim_{s\rightarrow t}\frac{f(t)-f(s)}{t-s}t^{1-\alpha }.
\end{equation}
(iv) If $f$ is fractional differentiable of order $\alpha$ at $t$,
then $f(\sigma (t))=f(t)+(\mu (t))t^{\alpha-1 }T_{\alpha}(f)(t)$.
\end{theorem}

\begin{proof}
$(i)$ Assume that $f$ is conformable fractional differentiable at $t$. Then,
there exists a neighborhood $\mathcal{V}_t$ of $t$ such that
$\left \vert \left[ f(\sigma (t))-f(s)\right] t^{1-\alpha }
-T_{\alpha}(f)(t)\left[ \sigma(t)-s\right] \right \vert
\leq \epsilon \left \vert \sigma(t)-s\right \vert$
for $s\in \mathcal{V}_t$. Therefore,
$\left \vert f\left( t\right) -f\left( s\right) \right \vert
\leq \left\vert \left[ f(\sigma (t)-f(s)\right]
-T_\alpha(f)(t)\left[ \sigma (t)-s\right]t^{\alpha -1}\right \vert
+ \left \vert \left[ f(\sigma (t))-f(t)\right]\right\vert
+ \left\vert f^{(\alpha)}(t)\right\vert \left\vert \left[ \sigma (t)-s
\right]\right\vert \left\vert t^{\alpha}-1\right\vert$
for all $s \in \mathcal{V}_t \cap \left]t-\epsilon ,t+\epsilon \right[$
and, since $t$ is a right-dense point,
\begin{equation*}
\left \vert f\left( t\right) -f\left( s\right) \right \vert
\leq \left \vert
\left[ f^{\sigma }(t)-f(s)\right] -f^{(\alpha )}(t)\left[ \sigma (t)
-s\right]^{\alpha }\right \vert +\left \vert f^{(\alpha )}(t)\left[
t-s\right]^{\alpha}\right\vert \\
\leq \epsilon \delta + \left \vert t^{\alpha-1}\right
\vert \left\vert T_\alpha(f)(t) \right\vert \delta.
\end{equation*}
Since $\delta \rightarrow 0$ when $s\rightarrow t$, and $t>0$,
it follows the continuity of $f$ at $t$.
$(ii)$ Assume that $f$ is continuous at $t$ and $t$ is right-scattered.
By continuity,
\begin{equation*}
\lim_{s\rightarrow t}\frac{f(\sigma(t))-f(s)}{\sigma (t)-s}t^{1-\alpha }
=\frac{f(\sigma(t))-f(t)}{\sigma (t)-t}t^{1-\alpha}
=\frac{f(\sigma(t))-f(t)}{\mu (t)}t^{1-\alpha }.
\end{equation*}
Hence, given $\epsilon >0$ and $\alpha \in ]0,1]$, there is a neighborhood
$\mathcal{V}_t$ of $t$ such that
\begin{equation*}
\left \vert \frac{f(\sigma)(t)-f(s)}{\sigma (t)-s}t^{1-\alpha }
-\frac{f(\sigma(t))-f(t)}{\mu (t)}t^{1-\alpha }\right \vert
\leq \epsilon
\end{equation*}
for all $s\in \mathcal{V}_t$. It follows that
\begin{equation*}
\left \vert \left[ f(\sigma)(t)-f(s)\right]t^{1-\alpha }
-\frac{f(\sigma(t))-f(t)}{\mu (t)}t^{1-\alpha }(\sigma (t)-s)\right \vert
\leq \epsilon |\sigma (t)-s|
\end{equation*}
for all $s\in \mathcal{V}_t$. The desired equality \eqref{eq:delf:dr}
follows from Definition~\ref{def:fd:ts}.
$(iii)$ Assume that $f$ is conformable fractional differentiable
of order $\alpha $ at $t$ and $t$ is right-dense.
Let $\epsilon > 0$ be given. Since $f$ is conformable fractional
differentiable of order $\alpha $ at $t$, there is a neighborhood
$\mathcal{V}_t$ of $t$ such that
$\left \vert \lbrack f(\sigma (t))-f(s)]t^{1-\alpha}
-T_{\alpha }(f)(t)(\sigma(t)-s)\right \vert
\leq \epsilon |\sigma (t)-s|$
for all $s\in \mathcal{V}_t$. Because $\sigma (t)=t$,
\begin{equation*}
\left \vert \frac{f(t)-f(s)}{t-s}t^{1-\alpha}-T_{\alpha}(f)(t)\right \vert
\leq \epsilon
\end{equation*}
for all $s\in \mathcal{V}_t$, $s\neq t$. Therefore, we get the desired result
\eqref{eq:rhsof}.
Now, assume that the limit on the right-hand side of \eqref{eq:rhsof}
exists and is equal to $L$, and $t$ is right-dense.
Then, there exists $\mathcal{V}_t$ such that
$\left \vert (f(t)-f(s))t^{1-\alpha }-L(t-s)\right \vert \leq \epsilon |t-s|$
for all $s\in \mathcal{V}_t$. Because $t$ is right-dense,
\begin{equation*}
\left \vert (f(\sigma(t))-f(s))t^{1-\alpha }
-L(\sigma(t)-s)\right \vert \leq \epsilon |\sigma(t)-s|,
\end{equation*}
which lead us to the conclusion that $f$ is conformable fractional differentiable
of order $\alpha $ at $t$ and $T_{\alpha}(f)(t)=L$.
$(iv)$ If $t$ is right-dense, i.e., $\sigma (t)=t$, then $\mu (t)=0$ and
$f(\sigma (t))=f(t)=f(t)+\mu (t)T_{\alpha}(f)(t)t^{1-\alpha }$.
On the other hand, if $t$ is right-scattered, i.e., $\sigma (t)>t$, then by $(iii)$
\begin{equation*}
f(\sigma(t))=f(t)+\mu (t)t^{\alpha -1}\cdot \frac{f(\sigma (t))-f(t)}{\mu
(t)}t^{1-\alpha}=f(t)+(\mu (t))^{\alpha-1 }T_{\alpha}(f)(t),
\end{equation*}
and the proof is complete.
\end{proof}

\begin{remark}
In a time scale $\mathbb{T}$, due to the inherited topology of
the real numbers, a function $f$ is always continuous
at any isolated point $t \in \mathbb{T}$.
\end{remark}

\begin{example}
\label{ex:1}
Let $h>0$ and $\mathbb{T}=h\mathbb{Z}:=\{hk\ :\ k\in
\mathbb{Z}\}$. Then $\sigma (t)=t+ h$
and $\mu(t)=h$ for all $t\in \mathbb{T}$.
For function $f: t\in \mathbb{T}\mapsto t^2\in \mathbb{R}$ we have
$T_\alpha(f)(t)= \left(t^2\right)^{(\alpha)}=(2t+h)t^{1-\alpha}$.
\end{example}

\begin{example}
Let $q>1$ and $\mathbb{T}=\overline{q^{\mathbb{Z}}}
:=q^{\mathbb{Z}}\cup \{0\}$
with $q^{\mathbb{Z}}:=\{q^k \ :\ k\in\mathbb{Z} \}$.
In this time scale
$$
\sigma (t)=\left\{\begin{array}{ll}
qt & \text{ if } t\neq 0\\
0& \text{ if } t=0
\end{array}\right.
\quad \text{and} \quad
\mu(t) =\left\{
\begin{array}{ll}
(q-1)t & \text{ if }  t\neq 0\\
0 & \text{ if } t=0.
\end{array}\right.
$$
Here $0$ is a right-dense minimum and every other
point in $\mathbb{T}$ is isolated. Now consider the
square function $f$ of Example~\ref{ex:1}. It follows that
\begin{equation*}
\begin{split}
T_\alpha(f)(t) = \left(t^2\right)^{(\alpha)}
&=\left\{\begin{array}{ll}
(q+1)t^{2-\alpha} & \text{ if } t\neq 0\\
0& \text{ if }  t=0.
\end{array}\right.
\end{split}
\end{equation*}
\end{example}

\begin{example}
Let $q>1$ and $ \mathbb{T}=q^{{\mathbb{N}}_0}
:=\{q^n : n\in{\mathbb{N}}_0 \}$.
For all $t\in \mathbb{T}$ we have
$\sigma (t)=qt$ and $\mu(t)=(q-1)t$.
Let $f:t\in \mathbb{T}\mapsto \log(t)\in \mathbb{R}$. Then
$T_\alpha(f)(t) = \left(\log(t)\right)^{(\alpha)}
=\frac{\log(q)}{(q-1)t^{\alpha}}$
for all $t\in \mathbb{T}$.
\end{example}

\begin{proposition}
\label{E1:i}
If $f:\mathbb{T}\rightarrow \mathbb{R}$ is defined by $f(t)= c$
for all $t\in\mathbb{T}$, $c\in \mathbb{R}$, then
$T_\alpha(f)(t) = (c)^{(\alpha)} = 0$.
\end{proposition}

\begin{proof}
If $t$ is right-scattered, then by Theorem~\ref{T1} (ii) one has
$T_{\alpha}(f)(t)=\frac{f(\sigma(t))-f(t)}{\mu(t)}t^{1-\alpha}=0$.
Otherwise, $t$ is right-dense and, by Theorem~\ref{T1} (iii),
$T_{\alpha}(f)(t) = \lim_{s \rightarrow t}\frac{c-c}{t-s}t^{1-\alpha} = 0$.
\end{proof}

\begin{proposition}
\label{E1:ii}
If $f:\mathbb{T}\rightarrow \mathbb{R}$ is defined by $f(t)=t$
for all $t\in \mathbb{T}$, then
\begin{equation*}
T_{\alpha}(f)(t) = (t)^{(\alpha)} =
\begin{cases}
t^{1-\alpha } & \textrm{ if } \alpha \neq 1, \\
1 & \textrm{ if } \alpha =1.
\end{cases}
\end{equation*}
\end{proposition}

\begin{proof}
From Theorem~\ref{T1} (iv), it follows that
$\sigma(t) = t + \mu(t)t^{\alpha-1} T_{\alpha}(f)(t)$,
$\mu(t) = \mu(t)t^{\alpha-1} T_{\alpha}(f)(t)$.
If $\mu(t) \ne 0$, then $T_{\alpha}(f)(t)=t^{1-\alpha}$
and the desired relation is proved.
Assume now that $\mu(t) = 0$, i.e., $\sigma(t) = t$. In this case
$t$ is right-dense and, by Theorem~\ref{T1} (iii),
$T_{\alpha}(f)(t) = \lim_{s \rightarrow t}\frac{t-s}{t-s}t^{1-\alpha}= t^{1-\alpha}$.
Therefore, if $\alpha =1$, then $T_{\alpha}(f)(t)=1$;
if $0<\alpha <1$, then $T_{\alpha}(f)(t)=t^{1-\alpha}$.
\end{proof}

Now, let us consider the two classical cases $\mathbb{T}=\mathbb{R}$
and $\mathbb{T}= h \mathbb{Z}$, $h > 0$.

\begin{corollary}
Function $f :\mathbb{R} \rightarrow \mathbb{R}$
is conformable fractional differentiable of order $\alpha$
at point $t \in \mathbb{R}$ if, and only if, the limit
$\lim_{s\rightarrow t}\frac{f(t)-f(s)}{t-s}t^{1-\alpha}$
exists as a finite number. In this case,
\begin{equation}
\label{KG:der}
T_{\alpha}(f)(t)=\lim_{s\rightarrow t}\frac{f(t)-f(s)}{t-s}t^{1-\alpha}.
\end{equation}
\end{corollary}

\begin{proof}
Here $\mathbb{T}=\mathbb{R}$, so all points are right-dense.
The result follows from Theorem~\ref{T1} (iii).
\end{proof}

\begin{remark}
The identity \eqref{KG:der}
corresponds to the conformable derivative introduced
in \cite{KhHYS} and further studied in \cite{Abdeljawad}.
\end{remark}

\begin{corollary}
Let $h > 0$. If $f :h\mathbb{Z} \rightarrow \mathbb{R}$, then
$f$ is conformable fractional differentiable of order $\alpha$
at $t\in h\mathbb{Z}$ with
$$
T_{\alpha}(f)(t) =\frac{f(t+h)-f(t)}{h}t^{1-\alpha}.
$$
\end{corollary}

\begin{proof}
Here $\mathbb{T}=h\mathbb{Z}$
and all points are right-scattered.
The result follows from Theorem~\ref{T1} (ii).
\end{proof}

Now we give an example using the time scale $\mathbb{T} = \mathbb{P}_{a,b}$,
which is a time scale with interesting applications in Biology \cite{F:C:biology}.

\begin{example}
Let $a, b > 0$ and consider the time scale ${\mathbb{P}}_{a,b}
= \displaystyle\bigcup_{k=0}^{\infty}[k(a + b),\ k(a + b) + a]$.
Then
$$
\sigma(t)=\left\{
\begin{array}{lcl}
t& & \mbox{if}\ t
\in \displaystyle\bigcup_{k=0}^{\infty} [k(a + b),\ k(a + b) + a),\\
t+b & & \mbox{if}\ t
\in \displaystyle\bigcup_{k=0}^{\infty} \{k(a + b) + a\}
\end{array} \right.
$$
and
$$
\mu(t)=\left\{ \begin{array}{lcl}
0& & \mbox{if}\ t \in
\displaystyle\bigcup_{k=0}^{\infty} [k(a + b),\ k(a + b) + a),\\
b & & \mbox{if}\ t \in
\displaystyle\bigcup_{k=0}^{\infty} \{k(a + b) + a\}.
\end{array} \right.
$$
Let $f : \mathbb{P}_{a,b} \rightarrow \mathbb{R}$ be continuous
and $\alpha \in \  ]0,1]$. It follows from Theorem~\ref{T1} that the
conformable fractional derivative of order $\alpha$ of a function
$f$ defined on $\mathbb{P}_{a,b}$ is given by
$$
T_{\alpha}(f)(t) =
\begin{cases}
\displaystyle \lim_{s \rightarrow t}
\frac{f(t)-f(s)}{(t-s)}t^{1-\alpha} & \text{ if } t \in
\displaystyle\bigcup_{k=0}^{\infty} [k(a + b),\ k(a + b) + a), \\[0.3cm]
\frac{f(t + b)-f(t)}{b}t^{1-\alpha} & \mbox{if}\ t\in
\displaystyle\bigcup_{k=0}^{\infty} \{k(a + b) + a\}.
\end{cases}
$$
\end{example}

For the conformable fractional derivative on time scales to be
useful, we would like to know formulas for the derivatives of sums,
products, and quotients of fractional differentiable functions.
This is done according to the following theorem.

\begin{theorem}
\label{T2}
Assume $f, g : \mathbb{T} \rightarrow \mathbb{R}$ are
conformable fractional differentiable of order $\alpha$.
Then,
\begin{enumerate}
\item[(i)] the sum $f+g:\mathbb{T}\rightarrow \mathbb{R}$
is conformable fractional differentiable with
$T_{\alpha}(f+g) = T_{\alpha}(f) + T_{\alpha}(g)$;

\item[(ii)] for any $\lambda \in \mathbb{R}$,
$\lambda f :\mathbb{T}\rightarrow \mathbb{R}$
is conformable fractional differentiable with
$T_{\alpha}(\lambda f)=\lambda T_{\alpha}(f)$;

\item[(iii)] if $f$ and $g$ are continuous, then
the product $f g :\mathbb{T}\rightarrow \mathbb{R}$ is conformable
fractional differentiable with
$T_{\alpha}(fg)=T_{\alpha}(f) g + (f \circ \sigma) T_{\alpha}(g)
= T_{\alpha}(f) (g\circ\sigma) + f T_{\alpha}(g)$;

\item[(iv)] if $f$ is continuous, then $1/f$
is conformable fractional differentiable with
$$
T_{\alpha}\left(\frac{1}{f}\right)
= -\frac{T_{\alpha}(f)}{f (f \circ \sigma)},
$$
valid at all points $t \in \mathbb{T}^{\kappa}$
for which $f(t)f(\sigma(t))\neq 0$;

\item[(v)] if $f$ and $g$ are continuous, then $f/g$
is conformable fractional differentiable with
$$
T_{\alpha}\left(\frac{f}{g}\right)
=\frac{T_{\alpha}(f) g - f T_{\alpha}(g)}{g (g\circ\sigma)},
$$
valid at all points $t \in \mathbb{T}^{\kappa}$
for which $g(t) g(\sigma(t))\neq 0$.
\end{enumerate}
\end{theorem}

\begin{proof}
Let us consider that $\alpha \in ]0,1]$, and let us assume that $f$
and $g$ are conformable fractional differentiable at $t
\in\mathbb{T}^{\kappa}$.
$(i)$ Let $\epsilon > 0$. Then there exist neighborhoods
$\mathcal{V}_{t}$ and $\mathcal{U}_{t}$ of $t$ for which
\begin{equation*}
\left|[f(\sigma(t))-f(s)]t^{1-\alpha}
-T_{\alpha}(f)(t)\left(\sigma(t)-s\right)\right|
\leq \frac{\epsilon}{2}|\sigma(t)-s|
\quad \text{ for all } s\in \mathcal{V}_{t}
\end{equation*}
and
\begin{equation*}
\left|[g(\sigma(t))-g(s)]t^{1-\alpha}
-T_{\alpha}(g)(t)(\sigma(t)-s)\right|
\leq \frac{\epsilon}{2}|\sigma(t)-s|
\quad \text{ for all } s\in \mathcal{U}_{t}.
\end{equation*}
Let $\mathcal{W}_{t}=\mathcal{V}_{t}\cap \mathcal{U}_{t}$. Then
$\bigl|[(f+g)(\sigma(t))-(f+g)(s)]t^{1-\alpha}-\left[T_{\alpha}(f)(t)
+T_{\alpha}(g)(t)\right](\sigma(t)-s)\bigr| \leq \epsilon |\sigma(t)-s|$
for all $s\in \mathcal{W}$. Thus, $f+g$ is conformable
differentiable at $t$ and $T_{\alpha}(f+g)(t)=T_{\alpha}(f)(t)+T_{\alpha}(g)(t)$.
$(ii)$ Let $\epsilon > 0$. Then
$\left|[f(\sigma(t))-f(s)]t^{1-\alpha}-T_{\alpha}(f)(t)(\sigma(t)-s)\right|
\leq \epsilon|\sigma(t)-s|$ for all $s$ in a neighborhood
$\mathcal{V}_{t}$ of $t$. It follows that
\begin{equation*}
\left|[(\lambda f)(\sigma(t))-(\lambda f)(s)]t^{1-\alpha} -\lambda
T_{\alpha}(f)(t)(\sigma(t)-s)\right|\leq \epsilon |\lambda| \,
|\sigma(t)-s| \text{ for  all } s \in \mathcal{V}_{t}.
\end{equation*}
Therefore, $\lambda f$ is conformable fractional differentiable at
$t$ and $T_{\alpha}(\lambda f)=\lambda T_{\alpha}(f)$ holds at $t$.\\
$(iii)$ If $t$ is right-dense, then
\begin{equation*}
\begin{split}
T_{\alpha}(fg)(t) &=\lim_{s\rightarrow
t}\left[\frac{f(t)-f(s)}{t-s}t^{1-\alpha}\right]g\left( t\right)
+\lim_{s\rightarrow t}\left[\frac{g(t)-g(s)}{t-s}t^{1-\alpha}\right]f\left( s\right)\\
&= T_{\alpha}(f)(t) g(t) + T_{\alpha}(g)(t) f(t)
= T_{\alpha}(f)(t) g(\sigma(t))+T_{\alpha}(g)(t) f(t).
\end{split}
\end{equation*}
If $t$ is right-scattered, then
\begin{equation*}
\begin{split}
T_{\alpha}\left( fg\right)(t)
&=\left[\frac{f(\sigma(t))-f(t)}{\mu
(t)}t^{1-\alpha}\right]g\left(\sigma(t)\right)
+
\left[\frac{g(\sigma(t))-g(t)}{\mu (t)}t^{1-\alpha }\right]f(t)\\
&=T_{\alpha}(f)(t) g(\sigma(t))+f(t)T_{\alpha}(g)(t).
\end{split}
\end{equation*}
The other product rule formula follows by interchanging
the role of functions $f$ and $g$.
$(iv)$ We use the conformable fractional derivative of a constant
(Proposition~\ref{E1:i}) and property $(iii)$ of Theorem~\ref{T2}
(just proved): from Proposition~\ref{E1:i} we know that
$T_{\alpha}\left(f \cdot \frac{1}{f}\right)(t)=(1)^{(\alpha)}=0$.
Therefore, by (iii)
$$
T_{\alpha}\left(\frac{1}{f}\right)(t) f(\sigma(t))
+T_{\alpha}(f)(t) \frac{1}{f(t)}=0.
$$
Since we are assuming $f(\sigma(t))\neq 0$,
$T_{\alpha}\left(\frac{1}{f}\right)(t)
=-\frac{T_{\alpha}(f)(t)}{f(t) f(\sigma(t))}$.
$(v)$ We use $(ii)$ and $(iv)$ to obtain
\begin{equation*}
T_{\alpha}\left(\frac{f}{g}\right)(t)=T_{\alpha}\left(f \cdot \frac{1}{g}\right)(t)
=f(t) T_{\alpha}\left(\frac{1}{g}\right)(t)+T_{\alpha}(f)(t)\frac{1}{g(\sigma(t))}
=\frac{T_{\alpha}(f)(t)g(t)-f(t)T_{\alpha}(g)(t)}{g(t)g(\sigma(t))}.
\end{equation*}
This concludes the proof.
\end{proof}

\begin{theorem}
\label{thm:der:pf}
Let $c$ be a constant, $m \in \mathbb{N}$, and
$\alpha \in \left] 0,1\right]$.
\begin{enumerate}
\item[(i)] If $f(t) =(t-c)^{m}$, then
$T_{\alpha}(f)(t)=t^{1-\alpha} \sum_{p = 0}^{m-1}
\left(\sigma(t)-c\right)^{m-1-p}(t-c)^{p}$.
\item[(ii)] If $g(t)=\frac{1}{(t-c)^{m}}$
and $(t-c)\left(\sigma(t)-c\right) \neq 0$, then
$T_{\alpha}(g)(t)=-t^{1-\alpha} \sum_{p = 0}^{m-1}
\frac{1}{(\sigma(t)-c)^{p+1}(t-c)^{m-p}}$.
\end{enumerate}
\end{theorem}

\begin{proof}
We prove the first formula by induction. If $ m=1$, then $f(t)=t-c$
and $T_\alpha(f)(t)=t^{1-\alpha}$ holds from Propositions~\ref{E1:i}
and \ref{E1:ii} and Theorem~\ref{T2} $(i)$. Now assume that
$$
T_{\alpha}(f)(t)=t^{1-\alpha}
\sum_{p = 0}^{m-1}(\sigma(t)-c)^{m-1-p}(t-c)^{p}
$$
holds for $f(t) =(t-c)^{m}$ and let $F(t)=(t-c)^{m+1}=(t-c)f(t)$. We
use Theorem~\ref{T2} $(iii)$ to obtain
$\left(F(t)\right)^{(\alpha)}=T_{\alpha}(t-c) f(\sigma(t)) + T_\alpha(f)(t)(t-c)
=t^{1-\alpha} \sum_{p = 0}^{m}(\sigma(t)-c)^{m-p}(t-c)^{p}$.
Hence, by mathematical induction, part $(i)$ holds.
(ii) Let $g(t)=\frac{1}{(t-c)^{m}}=\frac{1}{f(t)}$.
From $(iv)$ of Theorem~\ref{T2},
\begin{equation*}
g^{(\alpha)}(t)=-\frac{T_{\alpha}(f)(t)}{f(t)f(\sigma(t))}\\
=-t^{1-\alpha}\sum_{p = 0}^{m-1}
\frac{1}{(\sigma(t)-c)^{p+1}(t-c)^{m-p}},
\end{equation*}
provided $(t-c)\left(\sigma(t)-c\right) \neq 0$.
\end{proof}

We show some examples of application of Theorem~\ref{thm:der:pf}.

\begin{example}
Let $\alpha \in \left]0,1\right]$ and $f(t)=t^{m}$.
Then $T_{\alpha}(f)(t)=t^{1-\alpha}
\displaystyle\sum_{p = 0}^{m-1}\sigma(t)^{m-1-p}t^{p}$.
Note that if $t$ is right-dense, then $T_{\alpha}(f)(t)=mt^{m-\alpha}$.
If we choose $\mathbb{T} = \mathbb{R}$ and $\alpha=1$, then we obtain
the usual derivative: $T_{1}(f)(t)= mt^{m-1} = f'(t)$.
\end{example}

\begin{example}
Let $\alpha \in \left]0,1\right]$ and $f(t)=\frac{1}{t^m}$.
Then $T_{\alpha}(f)(t)= -t^{1-\alpha}
\displaystyle\sum_{p=0}^{m-1}\frac{1}{t^{p-m}\sigma(t)^{p+1}}$.
If $t$ is right-dense, then $T_{\alpha}(f)(t)= -\frac{m}{t^{m+\alpha}}$.
Moreover, if $\alpha=1$, then we obtain $T_{1}(f)(t)= -\frac{m}{t^{m+1}}$.
\end{example}

\begin{example}
If $f(t)=(t-1)^{2}$, then
$T_{\alpha}(f)(t)=t^{1-\alpha}\left[(\sigma(t)+1)^2
+(\sigma(t)+1)(t+1)+(t+1)^2\right]$
for all $\alpha \in \left]0,1\right]$.
\end{example}

The chain rule, as we know it from the classical differential calculus,
does not hold for the conformable fractional derivative on times
scales. This is well illustrated by the following example.

\begin{example}
\label{ex:conterex:cr}
Let $\alpha \in (0,\ 1)$; $\mathbb{T} = \mathbb{N} = \{1, 2, \ldots\}$,
for which $\sigma(t)=t+1$ and $\mu(t)=1$; and $f,\ g: \mathbb{T}\rightarrow
\mathbb{T}$ be given by $f(t)=g(t)=t$. Then,
$T_{\alpha}(f\circ g)(t) \ne T_{\alpha}(f)\left(g(t)\right) T_{\alpha}(g)(t)$:
$T_{\alpha}(f\circ g)(t)=t^{1-\alpha}$,
while
$T_{\alpha}(f)\left(g(t)\right) T_{\alpha}(g)(t) = t^{2(1-\alpha)}$.
\end{example}

We can prove, however, the following result.

\begin{theorem}[Chain rule]
\label{T3}
Let $\alpha \in \left]0,1\right]$. Assume
$g:\mathbb{T}\rightarrow \mathbb{R}$ is continuous and conformable
fractional differentiable of order $\alpha$ at $t \in
\mathbb{T}^{\kappa}$, and $f:\mathbb{R}\rightarrow \mathbb{R}$ is
continuously differentiable. Then there exists $c$ in the real
interval $[t,\sigma(t)]$ with
\begin{equation}
\label{q1}
T_\alpha(f\circ g)(t)=f'(g(c)) \, T_{\alpha}(g)(t).
\end{equation}
\end{theorem}

\begin{proof}
Let $t \in \mathbb{T}^{\kappa}$. First we consider
$t$ to be right-scattered. In this case,
$$
T_\alpha(f\circ g)(t)=\frac{f(g(\sigma(t)))-f(g(t))}{\mu(t)}t^{1-\alpha}.
$$
If $g(\sigma(t))= g(t)$, then we get $T_\alpha(f\circ g)(t)=0$ and
$T_{\alpha}(g)(t)=0$. Therefore, \eqref{q1} holds for any $c$ in the
real interval $[t,\sigma(t)]$. Now assume that $g(\sigma(t)) \neq g(t)$.
By the mean value theorem we have
\begin{equation*}
T_{\alpha}(f\circ g)(t)
=\frac{f(g(\sigma(t)))-f(g(t))}{g(\sigma(t))-g(t)}
\cdot \frac{g(\sigma(t))-g(t)}{\mu(t)}t^{1-\alpha}\\
=f'(\xi) T_{\alpha}(g)(t),
\end{equation*}
where $\xi \in ]g(t),\ g(\sigma(t))[$. Since
$g:\mathbb{T}\rightarrow \mathbb{R}$ is continuous, there is a
$c\in[t,\sigma(t)]$ such that $g(c)=\xi$, which gives the desired
result. Now let us consider the case when $t$ is right-dense. In this case
\begin{equation*}
\begin{split}
T_\alpha(f\circ g)(t)&=\lim_{s\rightarrow
t}\frac{f(g(t))-f(g(s))}{g(t)-g(s)} \cdot
\frac{g(t)-g(s)}{t-s}t^{1-\alpha}.
\end{split}
\end{equation*}
By the mean value theorem, there exist
$\xi_{s} \in ]g(t), g(\sigma(t))[$ such that
\begin{equation*}
\begin{split}
T_\alpha(f\circ g)(t)
&=\lim_{s\rightarrow t}\left\{f'(\xi_{s})
\cdot \frac{g(t)-g(s)}{t-s}t^{1-\alpha}\right\}.
\end{split}
\end{equation*}
By the continuity of $g$, we get that
$\displaystyle\lim_{s\rightarrow t}\xi_{s}=g(t)$. Then
$T_\alpha(f\circ g)(t)=f'(g(t)) \cdot T_\alpha(g)(t)$.
Since $t$ is right-dense, we conclude that $c=t=\sigma(t)$, which
gives the desired result.
\end{proof}

\begin{example}
Let $\mathbb{T}=2^\mathbb{N}$, for which $\sigma(t) = 2t$ and
$\mu(t)=t$. $(i)$ Choose $f(t)=t^{2}$ and $g(t)=t$.
Theorem~\ref{T3} guarantees that we can find a value $c$
in the interval $[t,\sigma(t)]=[t,2t]$, such that
\begin{equation}
\label{eq:ex:cr:fc}
T_{\alpha}(f\circ g)(t)=f'(g(c)) T_{\alpha}(g)(t).
\end{equation}
Indeed, from Theorem~\ref{T1} it follows that
$T_{\alpha}(f\circ g)(t)=3t^{1-\alpha}$,
$T_{\alpha}(g)(t)=t^{1-\alpha}$, and $f'(g(c))=2c$.
Equality \eqref{eq:ex:cr:fc} leads to
$3t^{1-\alpha} = 2ct^{1-\alpha}$
and so $c=\frac{3}{2}t \in [t, 2 t]$.
$(ii)$ Now let us take $f(t)=g(t)=t^2$ for all $t\in \mathbb{T}$.
We obtain $15t^{4-\alpha}=T_{\alpha}(f\circ g)(t)
=f'(g(c)) T_{\alpha}(g)(t) = 2c^23 t^{2-\alpha}$.
Therefore, $c=\sqrt{\frac{5}{2}}t \in [t, 2t]$.
\end{example}

To end Section~\ref{sub:sec:FD}, we consider conformable
derivatives of higher-order. More precisely,
we define the conformable fractional derivative $T_{\alpha}$
for $\alpha \in (n,\ n+1]$, where $n$ is some natural number.

\begin{definition}
\label{def:hofd}
Let $\mathbb{T}$ be a time scale, $\alpha\in (n,\ n+1]$, $n \in \mathbb{N}$,
and let $f$ be $n$ times delta differentiable at $t \in \mathbb{T}^{\kappa^n}$.
We define the conformable fractional derivative of $f$ of order $\alpha$ as
$T_{\alpha}(f)(t):=T_{\alpha-n}\left(f^{\Delta^{n}}\right)(t)$.
As before, we also use the notation $(f(t))^{(\alpha)} = T_{\alpha}(f)(t)$.
\end{definition}

\begin{example}
Let $\mathbb{T} = h \mathbb{Z}$, $h > 0$, $f(t) = t^3$, and
$\alpha=2.1$. Then, by Definition~\ref{def:hofd}, we have
$T_{2.1}(f)=T_{0.1}\left(f^{\Delta^2}\right)$.
Since $\sigma(t) = t+h$ and $\mu(t)= h$,
$T_{2.1}(f)(t)=\left(t^3\right)^{(2.1)}
=(6t+6h)^{(0.1)}$.
By Proposition~\ref{E1:i} and Theorem~\ref{T2} (i) and (ii),
we obtain that $T_{2.1}(f)(t)= 6 (t)^{(0.1)}$. We conclude
from Proposition~\ref{E1:ii} that $T_{2.1}(f)(t)= 6 t^{0.9}$.
\end{example}

\begin{theorem}
Let $\alpha \in (n,\ n+1]$, $n \in \mathbb{N}$.
The following relation holds:
\begin{equation}
\label{eq:cfdRDelta}
T_\alpha(f)(t)= t^{1+n-\alpha}f^{\Delta^{1+n}}(t).
\end{equation}
\end{theorem}

\begin{proof}
Let $f$ be a function $n$ times delta-differentiable.
For $\alpha \in (n, n+1]$, there exist $\beta\in (0, 1]$ such that
$\alpha= n + \beta$. Using Definition~\ref{def:hofd},
$T_\alpha(f) = T_{\beta}\left(f^{\Delta^{n}}\right)$.
From the definition of (higher-order) delta derivative
and Theorem~\ref{T1} $(ii)$ and $(iii)$, it follows that
$T_\alpha(f)(t)=t^{1-\beta}\left(f^{\Delta^{n}}\right)^{\Delta}(t)$.
\end{proof}


\section{Fractional Integration}
\label{sub:sec:FI}

Now we introduce the $\alpha$-conformable fractional integral
(or $\alpha$-fractional integral) on time scales.

\begin{definition}
\label{def:int}
Let $f:\mathbb{T}\rightarrow \mathbb{R}$ be a
regulated function. Then the $\alpha$-fractional integral of $f$,
$0<\alpha \leq 1$, is defined by
$\int f(t)\Delta^{\alpha}t := \int f(t)t^{\alpha -1}\Delta t$.
\end{definition}

\begin{remark}
For $\mathbb{T} = \mathbb{R}$ Definition~\ref{def:int} reduces
to the conformable fractional integral given in \cite{KhHYS};
for $\alpha = 1$ Definition~\ref{def:int} reduces
to the indefinite integral of time scales \cite{BP}.
\end{remark}

\begin{definition}
\label{def:intFracCauchy}
Suppose $f:\mathbb{T}\rightarrow \mathbb{R}$ is a regulated function.
Denote the indefinite $\alpha$-fractional integral of $f$ of order $\alpha$,
$\alpha \in (0,\ 1]$, as follows:
$F_{\alpha}(t)=\int f(t)\Delta^{\alpha} t$.
Then, for all $a, b\in \mathbb{T}$, we define
the Cauchy $\alpha$-fractional integral by
$\int_{a}^{b}f(t)\Delta^{\alpha} t
=F_{\alpha}(b)-F_{\alpha}(a)$.
\end{definition}

\begin{example}
Let $\mathbb{T}=\mathbb{R}$, $\alpha= \frac{1}{2}$, and $f(t)=t$. Then
$\int_{1}^{10^{2/3}}f(t)\Delta^{\alpha}t=6$.
\end{example}

\begin{theorem}
\label{thm:antiderivativ}
Let $\alpha \in (0,\ 1]$. Then, for any rd-continuous function
$f:\mathbb{T}\rightarrow \mathbb{R}$, there exist a function
$F_{\alpha}:\mathbb{T}\rightarrow \mathbb{R}$ such that
$T_{\alpha}\left(F_{\alpha}\right)(t)=f(t)$
for all $t\in \mathbb{T}^\kappa$. Function $F_{\alpha}$
is said to be an $\alpha$-antiderivative of $f$.
\end{theorem}

\begin{proof}
The case $\alpha=1$ is proved in \cite{BP}.
Let $\alpha\in (0,\ 1)$. Suppose $f$ is rd-continuous.
By Theorem~1.16 of \cite{BP1}, $f$ is regulated. Then,
$F_\alpha(t)=\int f(t)\Delta^{\alpha} t$ is conformable fractional
differentiable on $\mathbb{T}^\kappa$. Using
\eqref{eq:cfdRDelta} and Definition~\ref{def:int}, we obtain that
$T_\alpha\left(F_\alpha\right)(t)
=t^{1-\alpha} \left(F_{\alpha}(t)\right)^\Delta
=f(t)$, $t\in \mathbb{T}^\kappa$.
\end{proof}

\begin{theorem}
\label{Int-Proprty}
Let $\alpha\in (0,\ 1]$, $a, b, c \in \mathbb{T}$, $\lambda\in\mathbb{R}$,
and $f, g$ be two rd-continuous functions. Then,
\begin{enumerate}
\item[(i)] $\displaystyle\int_{a}^{b}[f(t)+g(t)]\Delta^{\alpha} t
= \displaystyle\int_{a}^{b}f(t)\Delta^{\alpha} t +
\displaystyle\int_{a}^{b}g(t)\Delta^{\alpha} t$;

\item[(ii)] $\displaystyle\int_{a}^{b}(\lambda f)(t)\Delta^{\alpha} t
= \lambda \displaystyle\int_{a}^{b}f(t)\Delta^{\alpha} t$;

\item[(iii)] $\displaystyle\int_{a}^{b}f(t)\Delta^{\alpha} t
= - \displaystyle\int_{b}^{a}f(t)\Delta^{\alpha} t$;

\item[(iv)] $\displaystyle\int_{a}^{b}f(t)\Delta^{\alpha} t
= \displaystyle\int_{a}^{c}f(t)\Delta^{\alpha} t +
\displaystyle\int_{c}^{b}f(t)\Delta^{\alpha} t$;

\item[(v)] $\displaystyle\int_{a}^{a}f(t)\Delta^{\alpha} t = 0$;

\item[(vi)] if there exist $g: \mathbb{T}\rightarrow \mathbb{R}$
with $|f(t)|\leq g(t)$ for all $t\in [a,\ b]$, then
$\left| \int_{a}^{b} f(t)\Delta^{\alpha} t\right|
\leq\int_{a}^{b} g(t)\Delta^{\alpha} t$;

\item[(vii)] if $f(t)>0$ for all $t\in [a,\ b]$, then
$\displaystyle\int_{a}^{b}f(t)\Delta^{\alpha} t\geq 0$.
\end{enumerate}
\end{theorem}

\begin{proof}
The relations follow from Definitions~\ref{def:int} and \ref{def:intFracCauchy},
analogous properties of the delta-integral, and the
properties of Section~\ref{sub:sec:FD} for the conformable
fractional derivative on time scales.
\end{proof}

\begin{theorem}
If $f: \mathbb{T}^\kappa \rightarrow \mathbb{R}$ is a
rd-continuous function and $t\in \mathbb{T}^\kappa$, then
$$
\int_t^{\sigma(t)}f(s)\Delta^{\alpha}s= f(t)\mu(t)t^{\alpha-1}.
$$
\end{theorem}

\begin{proof}
Let $f$ be a rd-continuous function on $\mathbb{T}^\kappa$. Then $f$ is
a regulated function. By Definition~\ref{def:intFracCauchy} and
Theorem~\ref{thm:antiderivativ}, there exist an antiderivative
$F_\alpha$ of $f$ satisfying
\begin{equation*}
\int_t^{\sigma(t)}f(s)\Delta^{\alpha}s
=F_\alpha(\sigma(t))-F_\alpha(t)
=T_\alpha\left(F_\alpha\right)(t) \mu(t) t^{1-\alpha}
=f(t) \mu(t) t^{1-\alpha}.
\end{equation*}
This concludes the proof.
\end{proof}

\begin{theorem}
Let $\mathbb{T}$ be a time scale, $a, b \in \mathbb{T}$ with $a < b$.
If $T_\alpha(f)(t)\geq 0$ for all $ t \in [a, b] \cap \mathbb{T}$,
then $f$ is an increasing function on $[a, b] \cap \mathbb{T}$.
\end{theorem}

\begin{proof}
Assume $T_\alpha(f)$ exist on $[a, b]\cap \mathbb{T}$ and
$T_\alpha(f)(t)\geq 0$ for all $t\in [a, b]\cap \mathbb{T}$.
Then, by (i) of Theorem~\ref{T1}, $T_\alpha(f)$ is continuous
on $[a, b]\cap \mathbb{T}$ and, therefore, by
Theorem~\ref{Int-Proprty} (vii),
$\int_{s}^{t}T_\alpha f(\xi)\Delta^{\alpha}\xi \geq 0$
for $s, t$ such that $a\leq s\leq t\leq b$.
From Definition~\ref{def:intFracCauchy},
$f(t)= f(s) + \int_{s}^{t}T_\alpha f(\xi)\Delta^{\alpha}\xi \geq f(s)$.
\end{proof}


\section{Conclusion}
\label{sec:conc}

A fractional calculus, that is, a study of differentiation and integration
of non-integer order, is here investigated via the recent and powerful
calculus on time scales. Our new calculus include, in a single theory,
discrete, continuous, and hybrid fractional calculi. In particular,
the new fractional calculus on time scales
unifies and generalizes: the Hilger calculus \cite{BP,H2},
obtained by choosing $\alpha = 1$; and the conformable
fractional calculus \cite{Abdeljawad,KhHYS,MR3326681},
obtained by choosing $\mathbb{T} = \mathbb{R}$.


\small


\section*{Acknowledgments}

This work was carried out while Nadia Benkhettou and Salima Hassani
were visiting the Department of Mathematics of University of Aveiro,
Portugal, February and March 2015. The hospitality of the host
institution and the financial support of Sidi Bel Abbes University,
Algeria, are here gratefully acknowledged. Torres was supported by
Portuguese funds through CIDMA and the Portuguese Foundation
for Science and Technology (FCT), within project UID/MAT/04106/2013.
The authors are grateful to two anonymous referees
for constructive comments and suggestions.



\end{document}